\documentclass[11pt,reqno]{article}
\usepackage{graphicx}
\usepackage{amssymb}
\usepackage{amsmath}
\usepackage{amsfonts}
\usepackage{bbm}
\usepackage{psfrag}
\usepackage{xcolor}
\usepackage{fullpage}
\usepackage{epstopdf}
\usepackage{graphicx}
\usepackage{times}
\usepackage{jmlr2e}
\usepackage{commath}
\usepackage{thmtools}
\usepackage{thm-restate}
\usepackage{times}
\usepackage{graphicx}
\usepackage{natbib}
\usepackage[colorlinks,citecolor=bbluegray,linkcolor=ddarkbrown,urlcolor=blue,breaklinks]{hyperref}

\definecolor{ddarkbrown}{rgb}{0.5,0.2,0.05} \definecolor{bbluegray}{rgb}{0.05,0,0.5}

\newcommand{\BEAS}{\begin{eqnarray*}}
\newcommand{\EEAS}{\end{eqnarray*}}
\newcommand{\BEA}{\begin{eqnarray}}
\newcommand{\EEA}{\end{eqnarray}}
\newcommand{\BEQ}{\begin{equation}}
\newcommand{\EEQ}{\end{equation}}
\newcommand{\BIT}{\begin{itemize}}
\newcommand{\EIT}{\end{itemize}}
\newcommand{\BNUM}{\begin{enumerate}}
\newcommand{\ENUM}{\end{enumerate}}

\newcommand{\BA}{\begin{array}}
\newcommand{\EA}{\end{array}}

\newcommand{\ie}{{ i.e.}} % FB: remove the \it for consistency in the paper

\newcommand{\reals}{{\mathbb R}}

\newcommand{\argmin}{\mathop{\rm argmin}}

\newcommand{\dom}{\mathop{\bf dom}}

\newcommand{\prox}{\text{\textbf{prox}}}

\newcounter{dummy} \numberwithin{dummy}{section}
\newtheorem{mythm}[dummy]{Theorem}
\newtheorem{mydef}[dummy]{Definition}
\newtheorem{myprop}[dummy]{Proposition}

\numberwithin{mythm}{section}
\numberwithin{mydef}{section}
\numberwithin{myprop}{section}

\newcommand{\scalar}[2]{\langle\,{#1,#2}\rangle}

\begin{document} 

\title{Integration Methods and Accelerated Optimization Algorithms}

\author{\name Damien Scieur  \email damien.scieur@inria.fr\\ 
\name Vincent Roulet \email vicnent.roulet@inria.fr\\ 
\name Francis Bach  \email francis.bach@ens.fr \\ 
\name Alexandre d'Aspremont \email aspremon@ens.fr \\
\addr
INRIA - Sierra Project-team \\
D\'epartement d'Informatique de l'Ecole Normale Sup\'erieure (CNRS - ENS - INRIA) \\
Paris, France}
\editor{}

\maketitle

\begin{abstract}
We show that accelerated optimization methods can be seen as particular instances of multi-step integration schemes from numerical analysis, applied to the gradient flow equation. In comparison with recent advances in this vein, the differential equation considered here is the basic gradient flow and we show that multi-step schemes allow integration of this differential equation using larger step sizes, thus intuitively explaining acceleration results.
\end{abstract}

\section*{Introduction}
\label{sec:intro}
The gradient descent algorithm used to minimize a function $f$ has a well-known simple numerical interpretation as the integration of the gradient flow equation, written
\begin{align}\tag{Gradient Flow} \label{eq:gradflow}
\begin{split}
x(0) & = x_0\\
\dot{x}(t) & = -\nabla f (x(t)),
\end{split}
\end{align}
using Euler's method. This appears to be a somewhat unique connection between optimization and numerical methods, since  these two fields have inherently different goals. On one hand, numerical methods aim to get a precise discrete approximation of the solution $x(t)$ on a finite time interval. More sophisticated methods than Euler's were developed to get better consistency with the continuous time solution but still focus on a finite time horizon~\citep[see for example][]{suli2003introduction}. On the other hand, optimization algorithms seek to find the minimizer of a function, which corresponds to the infinite time horizon of the gradient flow equation. Structural assumptions on $f$ led to more sophisticated algorithms than the gradient method, such as the mirror gradient method \citep[see for example][]{ben2001lectures,beck2003mirror}, proximal gradient method \citep{nesterov2007gradient} or a combination thereof \citep{duchi2010composite,nesterov2015universal}. Among them Nesterov's accelerated gradient algorithm \citep{nesterov1983method} is proven to be optimal on the class of smooth convex or strongly convex functions. This last method was designed with the lower complexity bounds in mind, but the proof relies on purely algebraic arguments and the key mechanism behind acceleration remains elusive, which led to various interpretations of it \citep{Bubeck15,AllenZhu2014LinearCA,lessard2016analysis}.

A recent stream of papers recently used differential equations to model the acceleration behavior and offer a better interpretation of Nesterov's algorithm \citep{su2014differential,krichene2015accelerated,wibisono2016variational,wilson2016lyapunov}.
However, the differential equation is often quite complex, being reverse-engineered from Nesterov's method itself, thus losing the intuition. Moreover, integration methods for these differential equations are often ignored or are not derived from standard numerical integration schemes, because the convergence proof of the algorithm does not require the continuous time interpretation. 

Here, we take another approach. Rather than using a complicated differential equation, we use advanced multi-step methods to discretize the basic gradient flow equation in~\eqref{eq:gradflow}. These lesser known methods, developed decades ago by numerical analysts, directly correspond to various well-known optimization algorithms. In particular, Nesterov's method can be seen as a stable and consistent gradient flow discretization scheme that allows bigger step sizes in integration, leading to faster convergence. 

The paper is organized as follows. In Section~\ref{sec:gradflow} we present our setting and recall classical results on differential equations. We then review definitions and theoretical properties of integration methods called linear multi-step methods in Section~\ref{sec:integrationOde}. Linear one-step and two-step methods are detailed in Section~\ref{sec:some_multistep}  and linked to optimization algorithms in Section~\ref{sec:links} (strongly convex case) and Section~\ref{sec:convex} (convex case). Finally, we propose several extensions for \eqref{eq:gradflow}, which can be used to explain proximal methods (Section \ref{sec:prox}), non-euclidean method (Section \ref{sec:noneuclidean}) or a combination thereof (Section \ref{sec:generalized}).

\section{Gradient flow}
\label{sec:gradflow}
We seek to minimize a $L$-smooth $\mu$-strongly convex function $f$ defined on $\reals^d$ and look in that purpose at the discretization the gradient flow equation \eqref{eq:gradflow}, given by the following ordinary differential equation (ODE)
\begin{align}
\begin{split}
\dot x(t) &= g(x(t))  \\
x(0) &= x_0 ,
\end{split}
\tag{ODE} \label{eq:ode}
\end{align}
where $g$ comes from a potential $-f$, meaning $g = -\nabla f$. Smoothness of $f$ means Lipschitz continuity of $g$, \ie~
\[
\|g(x)-g(y)\| \leq L\|x-y\|, \quad \mbox{for every $x,y\in\reals^d$},
\]
where $\|.\|$ is the Euclidean norm. This property ensures the existence and uniqueness of the solution of \eqref{eq:ode} (see \citep[Theorem 12.1]{suli2003introduction}). Strong convexity of $f$ means strong monotonicity of $-g$, \ie,
\[
\mu \|x-y\|^2 \leq -\langle x-y, g(x)-g(y)\rangle, \quad \mbox{for every $x,y\in\reals^d$},
\]
and ensures that~\eqref{eq:ode} has a unique point $x^*$ such that $g(x^*)= 0$, called the equilibrium. This is the minimizer of $f$ and the limit point of the solution,\ie,~$x(\infty) = x^*$. In the numerical analysis literature, strong monotonicity of $-g$ is referred to as one-sided Lipschitz-continuity of $g$. Whatever its name, this property essentially contracts solutions of~\eqref{eq:ode} with different initial points as showed in the following proposition.
\begin{myprop}
	Assume that $x_1(t),x_2(t)$ are solutions of \ref{eq:ode} with different initial conditions, where $-g$ is strongly monotone.
	Then, $x_1(t)$ and $x_2(t)$ get closer as $t$ increases, precisely,
	\[
	\|x_1(t)-x_2(t)\|^2 \leq e^{-2\mu t}\|x_1(0)-x_2(0)\|^2.
	\]
\end{myprop}
\begin{proof}
	Let $\mathcal{L}(t) = \|x_1(t)-x_2(t)\|^2$ If we derive $\mathcal{L}(t)$ over time,
	\BEAS
	\frac{\dif}{\dif t} \mathcal{L}(t) & = & 2\scalar{x_1(t)-x_2(t)}{\dot x_1(t)-\dot x_2(t)} \\
	& = & 2\scalar{x_1(t)-x_2(t)}{g(x_1(t))-g(x_2(t))}\\
	& \leq & -2\mu\mathcal{L}(t).
	\EEAS
	We thus have $\mathcal{L}(t) \leq e^{-2\mu t} \mathcal{L}(0)$, which is the desired result.
\end{proof}Strong convexity allows us to control the convergence rate of the potential $f$ and the solution $x(t)$ as recalled in the following proposition.
\begin{restatable}{myprop}{rate_cont_strconv}
	\label{prop:rate_cont_strconv} 
	Let $f$ be a $L$-smooth and $\mu$-strongly convex function and $x_0 \in \dom(f)$. Writing $x^*$ the minimizer of $f$, the solution $x(t)$ of \eqref{eq:gradflow} satisfies 
	\BEA
	f(x(t))-f(x^*) &\leq & (f(x_0)-f(x^*))e^{-2\mu t} \label{eq:rate_f_cont_strconv} 
	\\
	\|x(t) - x^*\| &\leq & \|x_0-x^*\| e^{-\mu t}. \label{eq:rate_normx_cont_strconv}
	\EEA
\end{restatable}
\begin{proof}
	Indeed, if we derive the left-hand-side of \eqref{eq:rate_f_cont_strconv},
	\BEAS
	\frac{\dif}{\dif t}[f(x(t))-f(x^*)] = \scalar{\nabla f(x(t))}{\dot x(t)} = -\|f'(x(t))\|^2.
	\EEAS
	Using that $f$ is strongly convex, we have (see \cite{nesterov2013introductory})
	\[
	f(x) - f(x^*) \leq  \frac{1}{2\mu} \|\nabla f(x)\|^2, 
	 \]
	and therefore
	\[
	\frac{\dif}{\dif t}[f(x(t))-f(x^*)] \leq -2 \mu [f(x(t))-f(x^*)].
	\]
	Solving this differential equation leads to the desired result. We can apply a similar technique for the proof of \eqref{eq:rate_normx_cont_strconv}, using that
	\[
	\mu \|x-y\|^2 \leq \scalar{\nabla f(x) - \nabla f(y)}{x-y}, 
	\]
	for strongly convex functions (see again \cite{nesterov2013introductory}).
\end{proof}
We now focus on numerical methods to integrate   \eqref{eq:ode}.

\section{Numerical integration of differential equations} 
\label{sec:integrationOde}

In general, we do not have access to an explicit solution $x(t)$ of~\eqref{eq:ode}. We thus use integration algorithms to approximate the curve $(t,x(t))$ by a grid $(t_k,x_k) \approx (t_k,x(t_k))$ on a finite interval $[0,t_{\max}]$. For simplicity here,  we assume the step size $h_k = t_{k}-t_{k-1}$ constant, i.e., $h_k = h$ and $t_k = kh$. 
The goal is then to minimize the approximation error $\|x_k-x(t_k)\|$ for $k \in [0,t_{\max}/h]$. We first introduce Euler's explicit method to illustrate this on a basic example.

\subsection{Euler's explicit method}
\label{ssec:euler_method}
Euler's explicit method is one of the oldest and simplest schemes for integrating the curve $x(t)$. The idea stems from the Taylor expansion of $x(t)$ which reads
\[
x(t+h) = x(t) + h\dot x(t) + O(h^2).
\]
When $t = kh$, Euler's explicit method approximates $x(t+h)$ by $x_{k+1}$ by neglecting the second order term,
\[
x_{k+1} = x_k + h g(x_k).
\]
In optimization terms, we recognize the gradient descent algorithm used to minimize $f$. Approximation errors in an integration method accumulate with iterations, and as Euler's method uses only the last point to compute the next one, it has only limited control over the accumulated error.

\subsection{Linear multi-step methods} \label{ssec:linear_multistep}
Multi-step methods use a combination of several past iterates to improve convergence. Throughout the paper, we focus on {\em linear} $s$-step methods whose recurrence can be written 
\BEQ
x_{k+s} = -\sum_{i=0}^{s-1}\rho_i x_{k+i} + h\sum_{i=0}^{s}\sigma_i g(x_{k+i}), \quad \mbox{for $k\geq 0$},
\EEQ
where $\rho_i, \sigma_i \in \reals$ are the parameters of the multi-step method and $h$ is again the step size. Each new point $x_{k+s}$ is a function of the information given by the $s$ previous points. If $\sigma_s = 0$, each new point is given explicitly by the $s$ previous points and the method is then called \textbf{explicit}. Otherwise each new point requires solving an implicit equation and the method is then called \textbf{implicit}.

To simplify notations we use the shift operator $E$, which maps $Ex_k \rightarrow x_{k+1}$. Moreover, if we write $g_k = g(x_k)$, then the shift operator also maps $Eg_k \rightarrow g_{k+1}$. Recall that a univariate polynomial is called monic if its leading coefficient is equal to 1. We now give the following concise definition of $s$-step linear methods.
\begin{mydef}
	A \textbf{linear $s$-step method} integrates an \eqref{eq:ode} defined by $g, x_0$, using a step size $h$ and $x_1,\ldots,x_{s-1}$ starting points by generating a sequence $x_k$ that satisfies
	\[
	\rho(E)x_k = h\sigma(E)g_k, \quad \mbox{for every $k\geq 0$},
	\]
	where $\rho$ is a monic polynomial of degree $s$ with coefficients $\rho_i$, and $\sigma$ a polynomial of degree $s$ with coefficients $\sigma_i$.
\end{mydef}

A linear $s-$step method is uniquely defined by the polynomials $(\rho,\sigma)$. The sequence generated by the method then depends on the starting points and the step size. 
Each linear multi step method has a twin sister, called \emph{one-leg} method, which generates the sequence $\tilde x_k$ following
\[
	\rho(E)\tilde x_k = \sigma(1) h g\left( \frac{\sigma(E)}{\sigma(1)}\tilde x_k \right).
\]
It is possible to show a bijection between $x_k$ and $\tilde x_k$ \citep{dahlquist1983one}. We quickly mention one-leg methods in Section \ref{sec:prox}  but we will not go into more details here. 
We now recall a few results describing the performance of multi-step methods.

\subsection{Stability}

Stability is a key concept for integration methods. First of all, consider two curves $x(t)$ and $y(t)$, both solutions of an \eqref{eq:ode} defined by $g$, but starting from different points $x(0)$ and $y(0)$. If the function $g$ is Lipchitz-continuous, it is possible to show that the distance between $x(t)$ and $y(t)$ is bounded on a finite interval, i.e.
\[
	\|x(t)-y(t)\| \leq C \|x(0)-y(0)\| \qquad \forall t \in [0,t_{\max}],
\]
where $C$ may depend exponentially on $t_{\max}$. We would like to have a similar behavior for our sequences $x_k$ and $y_k$, approximating  $x(t_k)$ and $y(t_k)$, i.e.
\BEA \label{eq:stability_continuous}
	\|x_{k}-y_k\| \approx 	\|x(t_k)-y(t_k)\| \leq C \|x(0)-y(0)\| \qquad \forall k \in [0,t_{\max}/h],
\EEA
when $h\rightarrow 0 $, so $k\rightarrow \infty$.
Two issues quickly arise, namely:
\begin{itemize}
	\item For a linear $s$-step method, we need $s$ starting values $x_{0},...,x_{s-1}$. Condition \eqref{eq:stability_continuous} will therefore depend on all these starting values and not only $x_0$. 
	\item Any discretization scheme introduce at each step an approximation error, called local error, which is accumulated over time. We denote this error by $\epsilon_{\mathrm{loc}}(x_{k+s})$ and define it as
	\[
		\epsilon_{\mathrm{loc}}(x_{k+s}) \triangleq x_{k+s}-x(t_{k+s})
	\]
if $x_{k+s}$ is computed using the real solution $x(t_{k}),...,x(t_{k+s-1})$.
\end{itemize}
In other words, the difference between $x_k$ and $y_k$ can be described as follows
\[
	\|x_{k}-y_k\|~ \leq ~\text{Error in the initial condition} + \text{Accumulation of local errors}.
\]
We now write a complete definition of stability, inspired from Definition 6.3.1 from \citet{gautschi2011numerical}.
\begin{mydef}
	A linear multi-step method is stable if, for two sequences $x_k$, $y_k$ generated by $(\rho,\sigma)$ using any sufficiently small step size $h>0$, from the starting values $x_0,...,x_{s-1}$, and $y_{0},...,y_{s-1}$, we have
	\BEQ
		\label{eq:def_stability}
		\|x_k-y_k\| \leq C \left( \max_{i \in \{0,...,s-1\}} \|x_i-y_i\| + \sum_{i=1}^{t_{\max}/h} \|\epsilon_{\mathrm{loc}}(x_{i+s})-\epsilon_{\mathrm{loc}}(y_{i+s})\| \right) ,
	\EEQ
for any $ k \in [0,t_{\max}/h]$. Here, the constant $C$ may depend on $t_{\max}$ but is independent of $h$. 
\end{mydef}
When $h$ tends to zero, we may recover equation \eqref{eq:stability_continuous} only if the accumulated local error tends also to zero. 
We thus need
\[
	\lim\limits_{h\rightarrow 0} \; \frac{1}{h} \; \|\epsilon_{\mathrm{loc}}(x_{i+s})-\epsilon_{\mathrm{loc}}(y_{i+s})\| = 0 \quad \forall i \in [0,t_{\max}/h].
\]
This condition is called \emph{consistency}. Once this condition satisfied, we still need to ensure
\BEQ \label{eq:zero-stability}
	\|x_k-y_k\| \leq C \max_{i \in \{0,...,s-1\}} \|x_i-y_i\|,
\EEQ
and this condition is called \emph{zero-stability}.

\subsubsection{Truncation error and consistency}\label{ssec:consistency}
The truncation error of a linear multi-step method is a measure of the local error $\epsilon_{\mathrm{loc}}(x_k)$ made by the method, normalized by $h$. More precisely it is defined using the difference between the step performed by the algorithm and the step which reaches exactly $x(t_{k+s})$, with
\BEQ\label{def:trunc}
T(h) \triangleq \frac{x(t_{k+s})-x_{k+s}}{h}  \qquad \text{assuming }x_{k+i} = x(t_{k+i}), \; i=0,\ldots,s-1.
\EEQ
This definition does not depend on $k$ but on the recurrence of the linear $s$-step method and on the \eqref{eq:ode} defined by $g$ and $x_0$. We can use this truncation error to define consistency. 
\begin{mydef}
An integration method for an \eqref{eq:ode} defined by $g, x_0$ is consistent if and only if, for any initial condition $x_0$,
	\[
	\lim\limits_{h\rightarrow 0} \|T(h)\| = 0.
	\]
\end{mydef}

The following proposition shows there exist simple conditions to check consistency, which rely on comparing a Taylor expansion of the solution with the coefficients of the method.

\begin{restatable}{myprop}{consistency} \label{prop:consistency}
	A linear multi-step method defined by polynomials $(\rho,\sigma)$ is consistent if and only if
	\BEQ
	\rho(1) = 0 \qquad \text{and} \qquad \rho'(1) = \sigma(1). \label{eq:consistency}
	\EEQ
\end{restatable}
\begin{proof}
	Assume $t_k = kh$. If we expand $g(x(t_k))$ we have
	\[
	g(x(t_k)) = g(x_0) + O(h).
	\]
	If we do the same thing with $x(t_k)$, we have
	\[
	x(t_k) = x_0 + kh \dot x(t_0) + O(h^2) = x_0 + kh g(x_0) + O(h^2).
	\]
	If we plug these results in the linear multi-step method,
\BEAS
	T(h) & = &\frac{1}{h}\left( x(t_{k+s}) +\sum_{i=0}^{s-1}\rho_i x(t_{k+i}) - h\sum_{i=0}^{s}\sigma_i g(x(t_{k+i})) \right) \\
	& = &\frac{1}{h} \rho(1) x_0 + (\rho'(1)-\sigma(1)) g(x_0) + O(h).
\EEAS
	The limit is equal to zero if and only if we satisfy \eqref{eq:consistency}.
\end{proof}
Consistency is crucial for integration methods, we give some intuition about how important are conditions defined in \eqref{eq:consistency}.

\paragraph{First condition, $\rho(1) = 0$.} If the condition is not satisfied, then the method exhibits an artificial gain or damping. Assume we start at some equilibrium $x^*$ of \ref{eq:ode} (i.e. $\nabla f(x^*) =0$), so $x_{i}=x^*$ for the first $s-1$ steps. The next iterate becomes
\[
x_s = -\sum_{i=0}^{s-1} \rho_{i} x^* + h \sigma(E) \underbrace{ g(x^*) }_{ =0 },
\]
and if $1+\sum_{i=0}^{s} \rho_{i} = \rho(1)\neq 0$, we have that the next iterate $x_s$ is different from $x^*$.

\paragraph{Second condition, $\rho'(1) = \sigma(1)$.} If this relation is not satisfied, we actually are integrating another equation than \eqref{eq:ode}. Assuming the first condition satisfied, $1$ is a root of $\rho$. Consider then the factorization
\[
\rho(z) = (z-1) \tilde \rho(z)
\]
where $\tilde \rho$ is a polynomial of degree $s-1$, and $\rho'(1) = \tilde \rho (1)$. The linear multi-step method becomes
\[
\tilde\rho(E) (y_{k+1}-y_k)  = h\sigma(E) g(y_k).
\]
If we sum up the above equation from the initial point, we get
\[
\tilde\rho(E) (y_{k}-y_0)  = \sigma(E) G_k,
\]
where $G_k = \sum_{i=0}^k hg(y_i)$. If $h$ goes to zero, our iterates $y_k$ converge to some continuous curve $c(t)$, and $G_k \rightarrow \int_{0}^{t} g(c(\tau)) \dif \tau$,
\[
\sum_{i=0}^s \tilde{\rho}_i (c(t)-x(0)) = \sum_{i=0}^{s-1} \sigma_i\int_{0}^{t} g(c(\tau)) \dif \tau.
\]
If we take the derivative over time, we get
\[
\tilde\rho(1)\dot c(t) = \sigma(1) g(c(t)) \quad
\Leftrightarrow \quad \rho'(1) \dot c(t) = \sigma(1) g(c(t)).
\]
which is different from the ODE we wanted to discretize, unless $\rho'(1)=\sigma(1)$.

\subsubsection{Zero-stability and root condition\label{ssec:zero_stability}}

To get stability, assuming consistency holds as above, we also need to satisfy the zero-stability condition \eqref{eq:zero-stability}, which characterizes the sensitivity of a method to initial conditions. Actually, the name comes from an interesting fact: analyzing the special case where $g = 0$ is completely equivalent to the general case, as stated in the \emph{root condition theorem}. The analysis becomes simplified, and reduces to standard linear algebra results because we only need to look at the solution of the homogeneous difference equation $\rho(E) x_k = 0$.
\begin{mythm}[Root condition] \label{theo:root_condition}
	Consider a linear multi-step method $(\rho,\sigma)$. The method is zero-stable if and only if all roots of $\rho$ are in the unit disk, and the roots on the unit circle are simple.
\end{mythm}
The proof of this theorem is technical and can be found as Theorem 6.3.4 of \citet{gautschi2011numerical}.

\subsection{Convergence of the global error and Dahlquist's equivalence theorem}\label{ssec:global_error}
Numerical analysis focuses on integrating an ODE on a finite interval of time $[0,t_{\max}]$. It studies the behavior of the global error defined by the difference between the scheme and the solution of the ODE, i.e., $x(t_k)-x_k$, as a function of the step size $h$. If the global error converges to $0$ with the step size, the method is guaranteed to approximate correctly the ODE on  the time interval, for $h$ small enough. The faster it converges with $h$, the less points we need to guarantee a global error within a given accuracy. 

We now state \emph{Dahlquist's equivalence theorem}, which shows that the global error converges to zero when $h$ does if the method is \emph{stable}, i.e., when the method is \emph{consistent} and \emph{zero-stable}. This naturally needs the additional assumption that the starting values $x_0,\ldots,x_{s-1}$ are computed such that they converge to the solution $(x(0),\ldots,x(t_{s-1}))$. 

\begin{mythm}[Dahlquist's equivalence theorem] \label{theo:equivalence_theo}
	Given an \eqref{eq:ode} defined by $g$ and $x_0$ and a consistent linear multi-step method $(\rho,\sigma)$, whose starting values are computed such that $\lim_{h \rightarrow 0 } x_i = x(t_i)$ for any $i \in \{0,\ldots,s-1\}$, zero-stability is necessary and sufficient for being convergent, i.e., $x(t_k)-x_k$ tends to zero for any $k$ when the step size $h$ tends to zero.
\end{mythm}
Again, the proof of the theorem can be obtained from \citet{gautschi2011numerical}. Notice that this theorem is fundamental in numerical analysis. For example, it says that if a \emph{consistent} method is not zero-stable, then the global error may be arbitrary large when the step size goes to zero, \emph{even if the local error decreases}. In fact, if zero-stability is not satisfied, there exists a sequence generated by the linear multi-step method which grows with arbitrarily large factors.

\subsection{Region of absolute stability} \label{ssec:region_absolute_stab}
Stability and global error are ensured on finite time intervals, however solving optimization problems requires us to look at the infinite time horizon.
We thus need more refined results and start by finding conditions ensuring that the numerical solution does not diverge when the time interval increases,\ie~that the numerical solution is stable with a constant $C$ which \emph{does not depend of $t_{\max}$}. Formally, for a fixed step-size $h$, we want to ensure
\BEQ\label{eq:stability}
\|x_k\| \leq C \max_{i \in \{0,...,s-1\}} \|x_i\| \quad \mbox{for all }  k\in[0,t_{\max}/h]\,~\mbox{and } t_{\max}>0.
\EEQ
This is not possible without assumptions on the function $g$ as in the general case the solution $x(t)$ itself may diverge. We begin with the simple scalar linear case which, given $\lambda>0$, reads 
\begin{align}
\begin{split}
\dot x(t) &= -\lambda x(t)  \\
x(0) &= x_0.
\end{split}
\tag{Scalar Linear ODE} \label{eq:scalar_linear_ode}
\end{align}
The recurrence of a linear multi-step methods with parameters $(\rho,\sigma)$ applied to \eqref{eq:scalar_linear_ode} then reads
\[
\rho(E)x_k = -\lambda h\sigma(E)x_k \quad \Leftrightarrow \quad [\rho+\lambda h \sigma](E)x_k = 0,
\]
where we recognize an homogeneous recurrence equation. 
Condition \eqref{eq:stability} is then controlled by the step size $h$ and the constant $\lambda$, ensuring that this homogeneous recurrent equation produces bounded solutions. This leads us to the definition of the region of absolute stability.
\begin{mydef}
	The region of absolute stability of a linear multi-step method defined by $(\rho,\sigma)$ is the set of values $\lambda h$ such that the characteristic polynomial 
\BEQ\label{eq:caracpoly}
\pi_{\lambda h} \triangleq \rho + \lambda h \sigma
\EEQ	
of the homogeneous recurrent equation $\pi_{\lambda h}(E)x_k =0$ produces bounded solutions.
\end{mydef}
Standard linear algebra links this condition to the roots of the characteristic polynomial as recalled in the next proposition (see Lemma 12.1 of \citet{suli2003introduction}).
\begin{myprop} \label{prop:conv_recurence}
	 Let $\pi$ be a polynomial and write $x_k$ a solution of the homogeneous recurrent equation $\pi(E)x_k = 0$ with arbitrary initial values. If all roots of $\pi$ are inside the unit disk and the ones on the unit circle have a multiplicity exactly equal to one, then $\|x_k\| \leq \infty$.
\end{myprop}
Absolute stability of a linear multi-step method determines its ability to integrate a linear ODE defined by 
\begin{align}
\begin{split}
\dot x(t) &= -A x(t)  \\
x(0) &= x_0,
\end{split}
\tag{Linear ODE} \label{eq:linear_ode}
\end{align}
where $A$ is a positive definite matrix whose eigenvalues belong to $[\mu,L]$ for $0<\mu \leq L$. In this case the step size $h$ must indeed be chosen such that for any $\lambda \in [\mu,L]$, $\lambda h$ belongs to the region of absolute stability of the method.
This \eqref{eq:linear_ode} is a special instance of \eqref{eq:gradflow} where~$f$ is a quadratic function. Therefore absolute stability gives necessary (but not sufficient) condition to integrate \eqref{eq:gradflow} of $L$-smooth $\mu$-strongly convex functions. 

\subsection{Convergence analysis in the linear case}
By construction, absolute stability also gives us hints on the convergence of $x_k$ to the equilibrium. More precisiely, it allows us to control the rate of convergence of $x_k$, approximating the solution $x(t)$ of \eqref{eq:linear_ode} (see Lemma 12.1 of \citet{suli2003introduction}).
\begin{myprop}\label{prop:conv_absolute_stab}
	Given a \eqref{eq:linear_ode} defined by $x_0$ and  a positive definite matrix $A$ whose eigenvalues belong to $[\mu,L]$ for $0<\mu \leq L$, for a fixed step size $h$ and a linear multi-step method defined by $(\rho,\sigma)$, let $r_{\max}$ be defined as
	\[
	r_{\max} = \max_{\lambda \in [\mu,L]} \; \max_{r \in \rm{roots}(\pi_{\lambda h}(z))} |r|, 
	\] 
where $\pi_{\lambda h} $ is defined in \eqref{eq:caracpoly}.
	If $r_{\max} <1 $ and its multiplicity is equal to $m$, then the speed of convergence of the sequence $x_k$ produced by the linear multi-step method to the equilibrium $x^*$ of the differential equation is given by
	\BEQ
		\|x_k-x^*\| = O(k^{m-1}r_{\max}^k).
	\EEQ
\end{myprop}
We can now use these properties to analyze and design multi-step methods.

\section{Analysis and design of multi-step methods}
\label{sec:some_multistep}
As shown before, we want to integrate \eqref{eq:gradflow} and Proposition \ref{prop:rate_cont_strconv} gives us a rate of convergence in the continuous case. If the method tracks $x(t)$ with sufficient accuracy, then the rate of the method will be close to the rate of convergence of $x(kh)$. So, \textit{larger values of $h$ yield faster convergence of $x(t)$ to the equilibrium $x^*$}. However $h$ cannot be too large, as the method may be too inaccurate and/or unstable as $h$ increases. {\em Convergence rates of optimization algorithms are thus controlled by our ability to discretize the gradient flow equation using large step sizes.} We recall the different conditions that proper linear multi-step methods should follow.

\begin{itemize}
	\item \emph{Monic polynomial (Section \ref{ssec:linear_multistep}).} This is a convention, otherwise dividing both sides of the difference equation of the multi-step method by $\rho_s$ does not change the method.
	\item \emph{Explicit method (Section \ref{ssec:linear_multistep}).} We assume that the scheme is explicit in order to avoid solving a non-linear system at each step (Section \ref{sec:prox} shows that implicit methods are linked to proximal methods).
	
	\item \emph{Consistency (Section \ref{ssec:consistency}).} If the method is not consistent, then the local error does not converge when the step size goes to zero.
	\item \emph{Zero-stability (Section \ref{ssec:zero_stability}).} Zero-stability ensures convergence of the global error (Section \ref{ssec:global_error}) when the method is also consistent.
	\item \emph{Region of absolute stability (Section \ref{ssec:region_absolute_stab}).} If $\lambda h$ is not inside the region of absolute stability for any $\lambda \in [\mu,L]$, then the method is divergent when $t_{\max}$ increases.
\end{itemize}
Using the remaining degrees of freedom, we will tune the algorithm to have the best rate of convergence on \eqref{eq:linear_ode}, which corresponds to the optimization of a quadratic function. Indeed, as showed in Proposition \ref{prop:conv_absolute_stab}, the largest root of $\pi_{\lambda h}$ gives us the rate of convergence on quadratic functions (when $\lambda \in [\mu,L]$). Since smooth and strongly convex functions are close to be quadratic (they are in fact sandwiched between two quadratics), this will also give us a good idea of the rate of convergence on these functions.

We do not derive a proof of convergence of the sequence for a general smooth and (strongly) convex function (in fact, it is already proved by \cite{nesterov2013introductory} or using Lyapunov techniques by \citet{wilson2016lyapunov}). But our results provide intuition on \emph{why} accelerated methods converge faster.

\subsection{Analysis and design of explicit Euler's method ($s=1$)}

In Section \ref{ssec:euler_method} we introduced Euler's method. In fact, we can view it as an explicit linear ``multi-step'' method with $s=1$ defined by the polynomials
\[
\rho(z) = -1 + z, \quad \sigma(z) = 1.
\]
We can check easily that it is consistent (using Proposition \ref{prop:consistency}) and zero-stable since $\rho$ has only one root which lies on the unit circle (Theorems \ref{theo:root_condition} and \ref{theo:equivalence_theo}). We need to determine the region of absolute stability in order to have an idea about the maximum value that $h>0$ can take before the method becomes unstable. Assume we want to integrate any $\mu$-strongly convex and $L$-smooth function $f$, with $0< \mu < L$ with any starting value $x_0$. Then, we need to find the set of values of $h$ such that the roots of the polynomial 
\[
\pi_{\lambda h}(z) = [\rho+\lambda h \sigma](z) = -1+\lambda h + z, \quad \lambda \in [\mu,L]
\]
are small. The unique root is $1-\lambda h$ and we need to solve the following minimax problem
\[
\min_{h}\max_{\lambda \in [\mu,L]}\left|1-\lambda h\right|,
\]
in the variable $h>0$.
The solution of this optimization problem is $h^* = \frac{2}{L+\mu}$, its optimal value is $(L-\mu)/(L+\mu)$ and its rate of convergence is then
\[
\|x_k-x^*\| = O\bigg(\Big(\frac{1-\mu/L}{1+\mu/L}\Big)^k\bigg).
\]
We recover the optimal step size and the rate of convergence of the gradient method for a general smooth and strongly convex function \citep{nesterov2013introductory}.

\subsection{Analysis of two-step methods ($s=2$)}
We will now analyze two-step methods. First we write the conditions of a good linear multi step method, introduced at the beginning of this section, into constraints on the coefficients.
\begin{align*}
\rho_2 &= 1 &\text{(Monic polynomial)}\\
\sigma_2 &= 0 &\text{(Explicit method)}\\
\rho_0+\rho_1+\rho_2 &= 0 &\text{(Consistency)}\\
\sigma_0+\sigma_1+\sigma_2  &= \rho_1 + 2\rho_2 &\text{(Consistency)}\\
|\text{Roots}(\rho)| &\leq 1 &\text{(Zero-stability).}
\end{align*}
If we use all equalities, we finally have three linear constraints, defined by
\begin{equation}
\mathcal{L} = \Big\{\rho_0,\,\rho_1,\,\sigma_1 : \rho_1  =   -(1+\rho_0); \quad \sigma_1 =  1-\rho_0-\sigma_0;\quad |\rho_0|  <  1\Big\}.
\label{eq:astab_set_linear}
\end{equation}
We will now try to find some condition on the remaining parameters in order to have a stable method. At first, let us analyze a condition on the roots of second order equations. Absolute stability requires that all roots of the polynomial $\pi_{\lambda h}$ are inside the unit circle.
The following proposition gives us the values of the roots of $\pi_{\lambda h}$ as a function of the parameters $\rho_i$ and $\sigma_i$.

\begin{restatable}{myprop}{absolute_stab_pilambdah}
	\label{prop:absolute_stab_pilambdah}
	Given constants $0 < \mu \leq L$, a step size $h>0$ and a linear two-step method defined by $(\rho,\sigma)$, under the conditions
	\BEAS
		(\rho_1+\mu h\sigma_1)^2 & \leq & 4(\rho_0+\mu h \sigma_0),\\
		(\rho_1+L h \sigma_1)^2 & \leq & 4(\rho_0+ L h \sigma_0),
	\EEAS
	the roots $r_{\pm}(\lambda)$ of $\pi_{\lambda h}$, defined in \eqref{eq:caracpoly}, are complex 
	for any $\lambda\in[\mu,L]$. Moreover, the largest modulus root is equal to
	\BEQ
		\max_{\lambda \in [\mu,L]} |r_{\pm}(\lambda)|^2  = \max\left\{ \rho_0 + \mu h \sigma_0,~ \rho_0 + L h \sigma_0 \right\} \label{eq:largest_modulus_root}.
	\EEQ
\end{restatable}
\begin{proof}
	We begin by analyzing the roots $r_\pm$ of the generic polynomial
	\[
	z^2+bz+c,
	\]
	where $b$ and $c$ are real numbers, corresponding to the coefficients of $\pi_{\lambda h}$, i.e. $b = \rho_1+\lambda h \sigma_1$ and $c = \rho_0+\lambda h \sigma_0$.  For a fixed $\lambda$ roots are complex if and only if 
	\[
	b^2 \leq 4c \quad \Leftrightarrow \quad (\rho_1+\lambda h\sigma_1)^2 - 4(\rho_0+\lambda h \sigma_0) \leq 0.
	\]
	Since the left side this is a convex function in $\lambda$, it is equivalent to check only for the extreme values
	\BEAS
	(\rho_1+\mu h\sigma_1)^2 & \leq & 4(\rho_0+\mu h \sigma_0),\\
	(\rho_1+L h \sigma_1)^2 & \leq & 4(\rho_0+ L h \sigma_0).
	\EEAS
	As roots are complex conjugates,
	\[
	|r_{\pm}(\lambda)|^2 = |c| = |\rho_0 + \lambda h \sigma_0|.
	\]
	Because the function is convex, the maximum is attained for an extreme value of $\lambda$,
	\BEAS
	\max_{\lambda \in [\mu,L]} |r_{\pm}(\lambda)|^2  = \max\left\{ \rho_0 + \mu h \sigma_0 ,~ \rho_0 + L h \sigma_0 \right\},
	\EEAS
	which is the desired result.
\end{proof}
The next step is to minimize the largest modulus defined in  \eqref{eq:largest_modulus_root} in the coefficients $\rho_i$ and $\sigma_i$ to get the best rate of convergence, assuming the roots are complex. We will not develop the case where the roots are real because this leads to weaker results.

\subsection{Design of optimal two-step method for quadratics}
We have now have all ingredients to build a two-step method for which the sequence $x_k$ converges quickly to $x^*$ for quadratic functions. We need to solve the following problem,
\[\BA{ll}
\mbox{minimize} &\max\left\{\rho_0 + \mu h \sigma_0,~\rho_0 + L h \sigma_0\right\} \\
\text{s.t.} & (\rho_0,\rho_1,\sigma_1) \in \mathcal{L} \\
& (\rho_1+\mu h\sigma_1)^2 \leq 4(\rho_0+\mu h \sigma_0) \\
& (\rho_1+L h \sigma_1)^2 \leq 4(\rho_0+ L h \sigma_0),
\EA\]
in the variables $\rho_0,\rho_1,\sigma_0,\sigma_1,h>0$,
where $\mathcal{L}$ is defined in \eqref{eq:astab_set_linear}. If we use the equality constraints in~\eqref{eq:astab_set_linear} and make the following change of variables,
\begin{equation}
\begin{cases}
\hat{h} & = h(1-\rho_0), \\
c_\mu & = \rho_0+ \mu h\sigma_0, \\
c_L & = \rho_0+ L h\sigma_0,
\end{cases}
\label{eq:astab_change_var}
\end{equation}
the problem becomes, for fixed $\hat h$,
\[\BA{ll}
\mbox{minimize}  & \max\left\{c_\mu,~c_L \right\} \\
\text{s.t.} & (-1-c_\mu +\mu \hat{h})^2 \leq 4 c_\mu \\
& (-1-c_L+L \hat{h} )^2 \leq 4 c_L \\
& |Lc_\mu - \mu c_L| < |L-\mu|,
\EA\]
in the variables $c_\mu, c_L$. In that case, the optimal solution is given  by
\BEQ \label{eq:optimal_roots}
	c_\mu^* = \Big(1-\sqrt{\mu \hat h}\Big)^2, \quad c_L^* = \Big(1-\sqrt{L \hat h}\Big)^2,
\EEQ
obtained by tightening the two first inequalities, for $\hat h \in ]0,\frac{(1+\mu/L)^2}{L}[$ such that last inequality is satisfied. Now if we fix $\hat{h}$ we can recover an optimal two step linear method defined by $(\rho,\sigma)$ and an optimal step size $h$ by using the equations in~\eqref{eq:astab_change_var}. We will use the following quantity 
\BEQ
\label{eq:beta}
\beta \triangleq \frac{1-\sqrt{\mu/L}}{1+\sqrt{\mu/L}}.
\EEQ

\paragraph{A suboptimal two-step method.}

We can fix $\hat h = 1/L$ for example. All computations done, the parameters of this two-step method, called method $\mathcal{M}_1$, are
\BEQ
\mathcal{M}_1 = 
\begin{cases}
	\rho(z) &= \beta -(1+\beta)z + z^2, \\
	\sigma(z) &= -\beta(1-\beta) + (1-\beta^2) z, \\
	h &= \frac{1}{L(1-\beta)},
\end{cases}
\label{eq:method_m1}
\EEQ
and its largest modulus root \eqref{eq:largest_modulus_root} is given by
\[
\text{rate}(\mathcal{M}_1) = \sqrt{\max\{ c_\mu ,~ c_L \}} = \sqrt{c_{\mu}} = 1-\sqrt{\mu/L}.
\]

\paragraph{Optimal two-step method for quadratics.} 

We can compute the optimal $\hat h$ which minimizes the maximum of the two roots $c^*_\mu$ and $c^*_L$ defined in~\eqref{eq:optimal_roots}. 
The solution is simply the one which balances the two terms in the maximum:
\[
\hat h^* = \frac{(1+\beta)^2}{L} \quad \Rightarrow \quad c_{\mu}^* = c_L^*.
\]
This choice of $\hat h$ leads to the method $\mathcal{M}_2$, described by
\BEQ
\mathcal{M}_2 = 
\begin{cases}
	\rho(z) &= \beta^2  -(1+\beta^2)z + z^2, \\
	\sigma(z) &= (1-\beta^2) z, \\
	h &= \frac{1}{\sqrt{\mu L}},
\end{cases}
\label{eq:method_m2}
\EEQ
with the rate of convergence
\[
\text{rate}(\mathcal{M}_2) = \sqrt{c_\mu} = \sqrt{c_L} = \beta < \text{rate}(\mathcal{M}_1).
\]
We will now see that methods $\mathcal{M}_1$ and $\mathcal{M}_2$ are actually related to Nesterov's method and Polyak's heavy ball algorithms.

\section{On the link between integration and optimization}
\label{sec:links}
In the previous section, we derived a family of linear multi-step methods, parametrized by $\hat h$. We will now compare these methods to common optimization algorithms used to minimize $L$-smooth, $\mu$-strongly convex functions.

\subsection{Polyak's heavy ball method}
The heavy ball method was proposed by \citet{polyak1964some}. It adds a momentum term to the gradient step
\[
x_{k+2} = x_{k+1} - c_1 \nabla f(x_{k+1}) + c_2 (x_{k+1}-x_k),
\]
where $c_1 = (1-\beta^2)/\sqrt{\mu L}$ and $c_2 = \beta^2 $
where $\beta$ is defined in (\ref{eq:beta}). We can organize the terms in the sequence to match the general structure of linear multi-step methods, to get
\[
\beta^2 x_k - (1+\beta^2) x_{k+1}  + x_{k+2} = (1-\beta^2)/\sqrt{\mu L} \left(-\nabla f(x_{k+1})\right).
\]
We easily identify $\rho(z) = \beta^2 -(1+\beta^2)z + z^2$ and $h\sigma(z) = (1-\beta^2)/\sqrt{\mu L} z$. To extract $h$, we will assume that the method is consistent (see conditions \eqref{eq:consistency}), which means
\BEAS
\rho(1) &=& 0  \qquad \qquad \text{Always satisfied}  \\ 
h\rho'(1) &=& h\sigma(1)  \qquad \Rightarrow \; h = \frac{1}{\sqrt{\mu L}}.
\EEAS
All computations done, we can identify the ``hidden'' linear multi-step method as
\BEQ
\mathcal{M}_{\text{Polyak}} = 
\begin{cases}
	\rho(z) & = \beta^2 -(1+\beta^2)z + 1\\
	\sigma(z) & = (1-\beta^2)z\\
	h & = \frac{1}{\sqrt{\mu L}}.
\end{cases}
\EEQ
This shows that $\mathcal{M}_{\text{Polyak}} = \mathcal{M}_2$. In fact, this result was expected since Polyak's method is known to be optimal for quadratic functions. However, it is also known that Polyak's algorithm does not converge for a general smooth and strongly convex function \citep{lessard2016analysis}.

\subsection{Nesterov's accelerated gradient}
Nesterov's accelerated method in its simplest form is described by two sequences $x_k$ and $y_k$, with
\BEAS
y_{k+1} & = & x_{k} - \frac{1}{L} \nabla f(x_k), \\
x_{k+1} & = & y_{k+1} + \beta(y_{k+1}-y_k).
\EEAS
As above, we will write Nesterov's accelerated gradient as a linear multi-step method by expanding $y_k$ in the definition of $x_k$, to get
\[
\beta x_k - (1+\beta) x_{k+1} + x_{k+2} = \frac{1}{L} \left( -\beta (-\nabla f(x_k)) + (1+\beta) (-\nabla f(x_{k+1})) \right).
\]
Consistency of the method is then ensured by 
\BEAS
\rho(1) &=& 0  \qquad \qquad \text{\text{Always satisfied}}\\
h\rho'(1) &=& h\sigma(1)  \qquad \Rightarrow \; h = \frac{1}{L(1-\beta)}.
\EEAS
After identification,
\BEAS
\mathcal{M}_{\text{Nest}} = 
\begin{cases}
	\rho(z) &= \beta -(1+\beta)z + z^2, \\
	\sigma(z) &= -\beta(1-\beta) + (1-\beta^2) z, \\
	h &= \frac{1}{L(1-\beta)},
\end{cases}
\EEAS
which means that $\mathcal{M}_1 = \mathcal{M}_{\text{Nest}}$. 

\subsection{Nesterov's method interpretation as a faster stable integration method}

\begin{figure}[t!]
	\centering
	\includegraphics[width=0.48\textwidth]{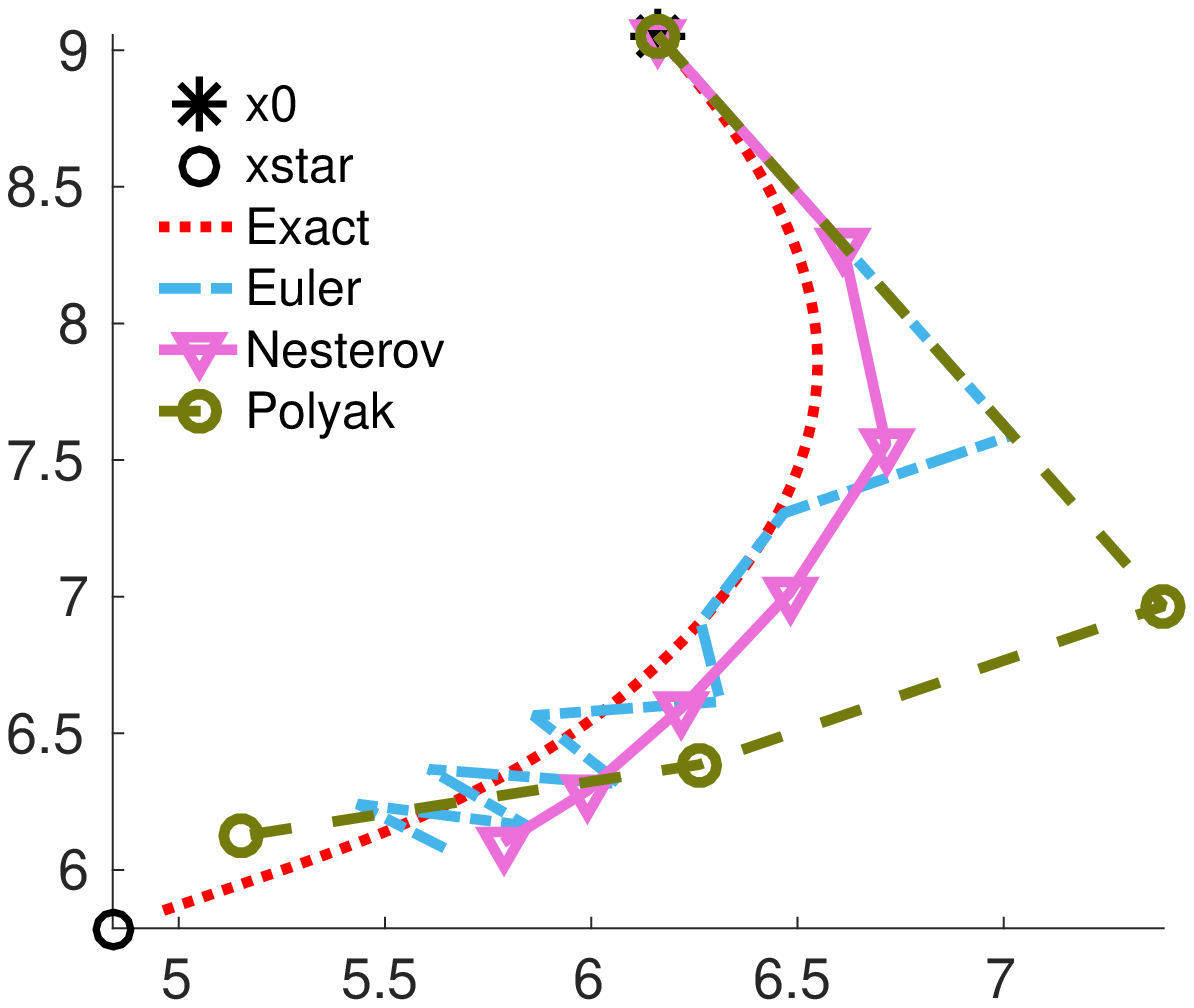}
	\includegraphics[width=0.48\textwidth]{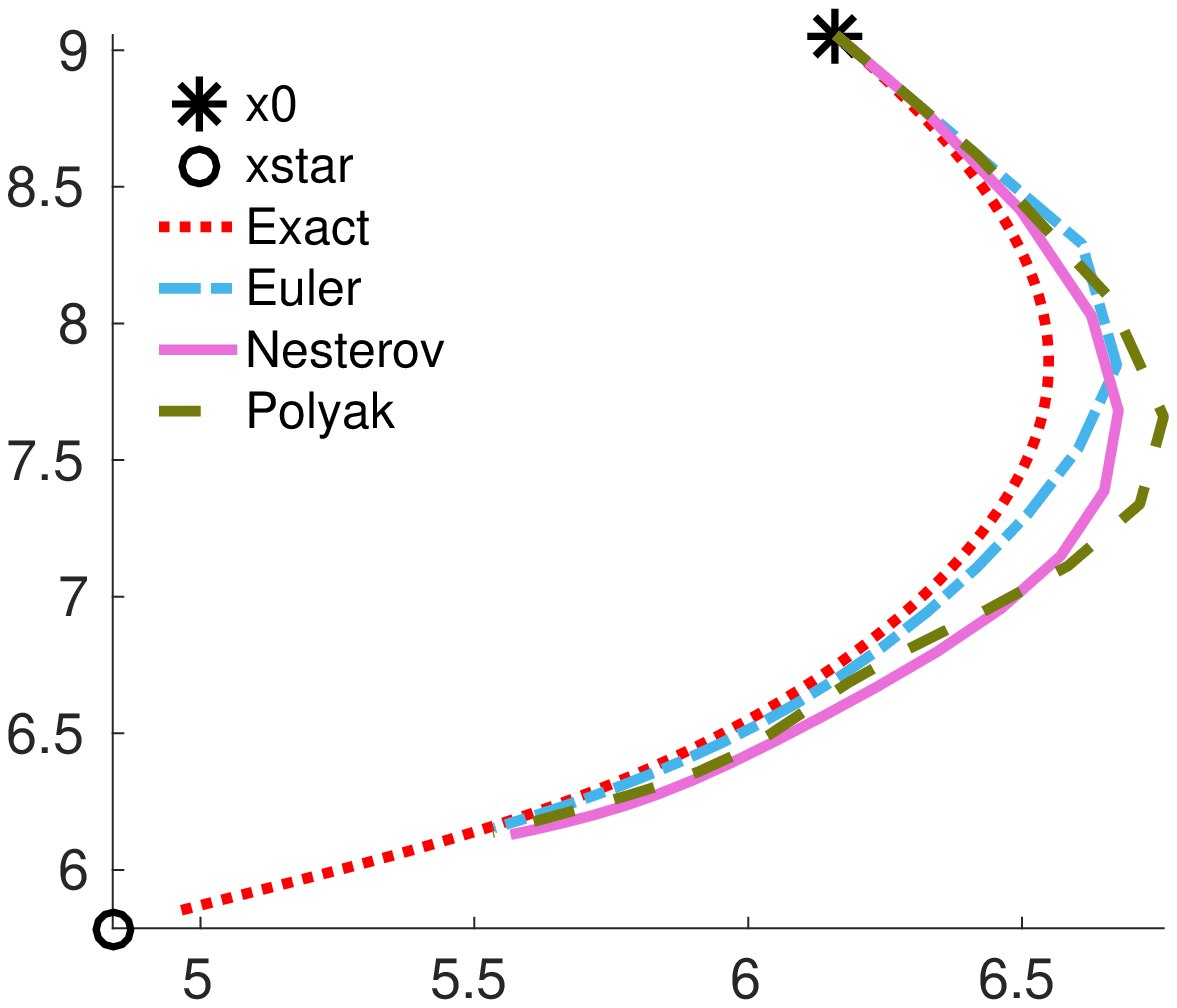}
	\caption{Integration of a \eqref{eq:linear_ode} (which corresponds to the minimization of a quadratic function) using Euler's, Nesterov's and Polyak's methods between $[0,t_{\max}]$. On the left, the optimal step size is used. Because Polyak's algorithm is the one with the biggest step size, it needs less iterations than Nesterov or Euler to approximate $x(t_{\max})$. On the right side, we reduced the step size to $1/L$ for all methods. We clearly observe that they all track the same ODE: the gradient flow.}
	\label{fig:integration}
\end{figure}

Pushing the analysis a little bit further, we can show \emph{why} Nesterov's algorithm is faster than gradient method. There is of course a complete proof of its rate of convergence \citep{nesterov2013introductory}, even using the argument of differential equations \citep{wibisono2016variational,wilson2016lyapunov}, but we take a more intuitive approach here. The key parameter is the step size $h$. If we compare it with the one in classical gradient method, Nesterov's method uses a step size which is $(1-\beta)^{-1} \approx \sqrt{L/\mu}$ larger.

Recall that, in continuous time, we have seen the rate of convergence of $x(t)$ to $x^*$ given by
\[
f(x(t))-f(x^*) \leq e^{-2\mu t} (f(x_0)-f(x^*)). 
\]
The gradient method tries to approximate $x(t)$ using Euler approximation with step size $h = 1/L$, which means $x_k^{(\text{grad})} \approx x(k/L)$, so
\[
f(x_k^{(\text{grad})}) - f(x^*) \approx f(x(k/L))-f(x^*) \leq (f(x_0)-f(x^*))e^{-2k\frac{\mu}{L}}.
\]
However, Nesterov's method has the step size
\[
h_{\text{Nest}} = \frac{1}{L(1-\beta)} = \frac{1+\sqrt{\mu/L}}{2\sqrt{\mu L}} \approx \frac{1}{\sqrt{4\mu L}} \qquad \text{which means} \quad x_{k}^{\text{nest}} \approx x\left(k/\sqrt{4\mu L}\right).
\]
In that case, the estimated rate of convergence becomes
\[
f(x_{k}^{\text{nest}})-f(x^*) \approx f(x(k/\sqrt{4\mu L})) - f(x^*) \leq (f(x_0)-f(x^*))e^{-k\sqrt{\mu/L}},
\]
which is approximatively the rate of convergence of Nesterov's algorithm in discrete time and we recover the accelerated rate in $\sqrt{\mu/L}$ versus $\mu/L$ for gradient descent. The accelerated method is more efficient because it integrates the gradient flow \emph{faster} than simple gradient descent, making longer steps. A numerical simulation in Figure \ref{fig:integration} makes this argument more visual. This intuitive argument is still valid for the convex counterpart of Nesterov's accelerated gradient.

\section{Acceleration for convex functions}
\label{sec:convex}
By matching the coefficients of Nesterov's method, we deduced the value of the step-size used for the integration of \eqref{eq:gradflow}. Then, using the rate of convergence of $x(t)$ to $x^*$, we estimated the rate of convergence of Nesterov's method assuming $x_{k}\approx x(t_k)$. Here, we will do the same but without assuming strong convexity. However, the estimation of the rate of convergence in discrete time needs the one in continuous time, described by the following proposition.
\begin{restatable}{myprop}{rate_convex}\label{prop:rate_convex}
	Let $f$ be $L$-smooth convex function, $x^*$ one of its minimizers and $x(t)$ be the solution of \eqref{eq:gradflow}. Then
	\BEQ \label{eq:rate_conv_convex}
		f(x(t))-f(x^*) \leq \frac{\|x_0-x^*\|^2}{t+(2/L)}.
	\EEQ
\end{restatable}
\begin{proof}
	Let $\mathcal{L}(x(t))=f(x(t))-f(x^*)$. We notice that $\nabla \mathcal{L}(x(t)) =\nabla f(x(t)) $ and $\dif \mathcal{L}(x(t)) / \dif t  = -\|\nabla f(x(t))\|^2 $. Since $f$ is convex, 
	\BEAS
	\mathcal{L}(x(t)) & \leq &\scalar{\nabla f(x(t))}{x(t)-x^*}\\
	& \leq & \|\nabla f(x(t))\|\|x(t)-x^*\|.
	\EEAS
	By consequence,
	\BEQ
	-\|\nabla f(x(t))\|^2 \leq -\frac{\mathcal{L}(x(t))^2}{\|x(t)-x^*\|^2} \leq -\frac{\mathcal{L}(x(t))^2}{\|x_0-x^*\|^2}.\label{eq:ratio_lyapunov_norm}
	\EEQ
	The last inequality comes from the fact that $\|x(t)-x^*\|$ decreases over time,
	\BEAS
	\frac{\dif}{\dif t}\|x(t)-x^*\|^2 &=& 2 \scalar{\dot x(t)}{x(t)-x^*}, \\
	& = & -2 \scalar{\nabla f(x(t))}{x(t)-x^*}, \\
	& \leq & 0 \quad \text{since $f$ is convex.}
	\EEAS	
	From \eqref{eq:ratio_lyapunov_norm}, we deduce the differential inequality
	\[
	\frac{\dif}{\dif t}  \mathcal{L}(x(t)) \leq -\frac{\mathcal{L}(x(t))^2}{\|x_0-x^*\|^2}.
	\]
	The solution is obtained by integration, 
	\[
	\int_0^t\frac{\dif \mathcal{L}(x(\tau)) / \dif \tau}{\mathcal{L}(x(\tau))^2} \dif \tau \leq \int_{0}^t \frac{-1}{\|x_0-x^*\|^2}.
	\]
	The general solution is thus
	\[ 
	\mathcal{L}(x(t)) \leq \frac{1}{\frac{t}{\|x_0-x^*\|^2} + C},
	\] 
	for some constant $C$. Since the inequality is valid for all time $t\geq 0$, the following condition on $C$,
	\[
	\mathcal{L}(x(t)) \leq \frac{1}{\frac{t}{\|x_0-x^*\|^2} + C} \leq \frac{1}{C} \quad \text{ for } \quad t\geq 0,
	\]
	is sufficient. Setting $C = \frac{1}{f(x_0)-f(x^*)}$ satisfies the above inequality. Using smoothness of $f$,
	\[
	f(x_0)-f(x^*)\leq \frac{L}{2}\|x_0-x^*\|^2,
	\]
	we get the desired result.
\end{proof}
Assume we use Euler's method with step size $h = \frac{1}{L}$, the estimated rate of convergence will be
\[
	f(x_{k}) - f(x^*) \approx f(x(kh))-f(x^*) \leq  \frac{L\|x_0-x^*\|^2}{k+2},
\]
which is close to the rate of convergence of the classical gradient method for convex function. Now, consider Nesterov's method for minimizing a smooth and convex function $f$:
\BEAS
	y_{k+1} & = & x_k-\frac{1}{L}\nabla f(x_k) \\
	x_{k+1} & = & -\beta_k x_k + (1+\beta_k) x_{k+1},
\EEAS
where $\beta_k \approx \frac{k-2}{k+1}$. If we expand everything, we get after rearrangement,
\[
	\beta_k x_{k-1} - (1+\beta_k) x_k + x_{k+1} = \frac{1}{L}\left( \beta_k (-\nabla f(x_{k-1})) - (1+\beta_k)(-\nabla f(x_k))  \right).
\]
In other terms, we have an expression of the form $\rho_k(E) x_k = h_k\sigma_k(E) (-\nabla f(x_k))$. We can identify $h$ if we assume the method consistent, which means
\BEAS
\rho(1) &=& 0  \qquad \qquad \text{Always satisfied}  \\
h_k\rho'_k(1) &=& h_k\sigma_k(1)  \qquad \Rightarrow \; h_k = \frac{1}{L(1-\beta_{k+1})} = \frac{(k+2)}{3L}. 
\EEAS
We can estimate, using \eqref{eq:rate_conv_convex}, the rate of convergence of Nesterov's method. Since $x_{k} \approx x(t_k)$,
\[
	x_k \approx x\big({\textstyle\sum_{i=0}^kh_i}\big) \approx x\big( {\textstyle\frac{k^2}{6L}} \big).
\]
In terms of convergence to the optimal value,
\[
	f(x_k)-f(x^*) \approx f(x(t_k)-f(x^*) \leq \frac{6L\|x_0-x^*\|^2}{k^2+12},
\]
which is close to the bound from \citet{nesterov2013introductory}. Again, because the step-size of Nesterov's algorithm is larger (while keeping a stable sequence), we converge faster than the Euler's method.

\section{Proximal algorithms and implicit integration methods}
\label{sec:prox}
We present here links between proximal algorithms and implicit numerical methods that integrate the gradient flow equation. We begin with Euler's implicit method that corresponds to the proximal point algorithm. 

\subsection{Euler's implicit method and proximal point algorithm}
We saw in Section \ref{ssec:euler_method} that Euler's explicit method used the Taylor expansion of the solution $x(t)$ of the \eqref{eq:ode} at the current point. The implicit version uses the Taylor expansion at the next point which reads
\[
x(t) = x(t+h) -h\dot{x}(t+h) +O(h^2).
\]
If $t = kh$, by neglecting the second order term we get implicit Euler's method,
\BEQ\label{eq:implicit}
x_{k+1} = x_k + hg(x_{k+1}).
\EEQ
This recurrent equation requires to solve an implicit equation at each step that may be costly. However it provides better stability than the explicit version. This is generally the case for implicit methods (see \citet{suli2003introduction} for further details on implicit methods). 

Now assume that $g$ comes from a potential $-f$ such that we are integrating \eqref{eq:gradflow}. Solving the implicit equation \eqref{eq:implicit} is equivalent to compute the proximal operator of $f$ defined as
\BEQ\label{eq:prox}
	\prox_{f,h}(x) = \argmin_z \frac{1}{2}\|z-x\|^2_2 + hf(z).
\EEQ
This can be easily verified by checking the first-order optimality conditions of the minimization problem. 
Euler's implicit method applied to \eqref{eq:gradflow} reads then 
\[
x_{k+1} = \prox_{f,h}(x_k),
\]
where we recognize the proximal point algorithm \citep{rockafellar1976monotone}.

We present now Mixed ODE that corresponds to composite optimization problems. 

\subsection{Implicit Explicit methods and proximal gradient descent}
In numerical analysis, it is common to consider the differential equation
\BEQ
\dot x = g(x) + \omega(x), \tag{Mixed ODE} \label{eq:mixed_ode}
\EEQ
where $g(x)$ is considered as the ``non-stiff'' part of the problem and $\omega$ the stiff one, where stiffness may be assimilated to bad conditioning \citep{ascher1995implicit,frank1997stability}.
Usually, we assume $\omega$ integrable using an implicit method. If $\omega$ derives from a potential $-\Omega$ (meaning $\omega = -\nabla \Omega$), this is equivalent to assume that the proximal operator of $\Omega$ defined in \eqref{eq:prox} can be computed exactly.

We approximate the solution of \eqref{eq:mixed_ode} using IMplicit-EXplicit schemes (IMEX). In our case, we will focus on the following multi-step based IMEX scheme,
\[
\rho(E)x_k = h\big(\sigma(E)g(x_k) + \gamma(E)\omega(x_k)\big),
\]
where $\rho, \sigma$ and $\gamma$ are polynomials of degrees $s$, $s-1$ (the explicit part) and $s$ respectively and $\rho$ is monic. It means that, at each iteration, we need to solve, in $x_{k+s}$,
\[
x_{k+s} = \sum_{i=0}^{s-1} \underbrace{\left(-\rho_i x_{k+i} + \sigma_i hg(x_{k+i}) + \gamma_ih\omega(x_{k+i}) \right)}_{\text{known}} + \gamma_s \omega(x_{k+s}).
\]

In terms of optimization the mixed ODE corresponds to composite minimization problems of the form 
\BEQ\label{eq:composite}
\mbox{minimize}\quad  f(x) + \Omega(x),
\EEQ
where $f, \Omega$ are convex and $\Omega$ has a computable proximal operator. We can link IMEX schemes with many optimization algorithms which use the proximal operator, such as proximal gradient method, FISTA or Nesterov's method. For example, proximal gradient is written
\BEAS
y_{k+1} & = & x_k - h \nabla f(x_k)\\
x_{k+1} & = & \prox_{h\Omega}(y_{k+1}).
\EEAS
After expansion, we get
\[
x_{k+1} = y_{k+1} - h \nabla \Omega(x_{k+1}) = x_k + h g(x_k)+h\omega(x_{k+1}),
\]
which corresponds to the IMEX method with polynomials
\[
\rho(z) = -1+z,\quad \sigma(z) = 1,\quad \gamma(z) = z.
\]
However, for Fista and Nesterov's method, we need to use a variant of linear multi-step algorithms, called \emph{one leg} methods \citep{dahlquist1983one,zhang2016stability}. Instead of combining the gradients, the idea is to compute $g$ at a linear combination of the previous points, i.e.
\[
\rho(E) x_k = h\left( g(\sigma(E)x_k) + \omega(\gamma(E)x_k) \right).
\]
Their analysis (convergence, consistency, interpretation of $h$, etc...) is slightly different from linear multi-step method, so we will not go into details in this paper, but the correspondence still holds.

\subsection{Non-smooth gradient flow \label{ssec:nonsmooth_gradflow}}
In the last subsection we assumed that $\omega$ comes from a potential. However in the optimization literature, composite problems have a smooth convex part and a non-smooth sub-differentiable convex part which prevents us from interpreting the problem with the gradient flow ODE. 
Non-smooth convex optimization problems can be treated with differential inclusions (see \citep{BolteDL07} for recent results on it) 
\[
	\dot{x}(t) + \partial f(x(t))  \ni 0,
\]
where $f$ is a sub-differentiable function whose sub-differential at $x$ is written $\partial f(x)$. Composite problems \eqref{eq:composite} can then be seen as the discretization of the differential inclusion
\[
	\dot{x}(t) + \nabla f(x(t)) + \partial \Omega x(t) \ni 0.
\]

\section{Mirror gradient descent and non-Euclidean gradient flow}
\label{sec:noneuclidean}
In many optimization problems, it is common to replace the Euclidean geometry with a distance-generating function called $d(x)$, with the associated Bregman divergence
\[
\mathcal{B}_d(x,y) = d(x)-d(y)-\scalar{\nabla d(y)}{x-y},
\]
with $d$ strongly-convex and lower semi-continuous. To take into account this geometry we consider the Non-Euclidean Gradient Flow \citep{krichene2015accelerated}
\begin{align}
\begin{split}
\dot y(t) &= -\nabla f \left( x(t) \right)  \\
x(t) &= \nabla d^* (y(t))\\
x(0) &= x_0, \; y(0) = \nabla d(x_0).
\end{split}
\tag{NEGF} \label{eq:negf}
\end{align}
Here $\nabla d$ maps primal variables to dual ones and, as $d$ is strongly convex, $(\nabla d)^{-1} = \nabla d^*$, where $d^*$ is the Fenchel conjugate of $d$. In fact, we can write \eqref{eq:negf} using only one variable $y$, but this formulation has the advantage to exhibit both primal and dual variables $x(t)$ and $y(t)$. Applying the forward Euler's explicit method we get the following recurrent equation
\[
y_{k+1} - y_k = - h \nabla f ( x_k),\quad x_{k+1} = \nabla d^* y_{k+1}.
\]
Now consider the mirror gradient scheme :
\[
x_{k+1} = \argmin_x \; h\scalar{\nabla f(x_k)}{x} + \mathcal{B}_h(x,x_k).
\]
First optimality condition reads
\[
\nabla_{x} \left( h\scalar{\nabla f(x_k)}{x} + \mathcal{B}_h(x,x_k)\right)\big|_{x=x_{k+1}} = h\nabla f(x_{k}) + \nabla d(x_{k+1}) - \nabla d(x_{k}) = 0
\]
Using that $(\nabla d)^{-1} = \nabla d^*$ we get
\[
h\nabla f( x_k ) + y_{k+1} - y_{k} = 0, \quad x_{k+1} = \nabla d^* y_{k+1},
\]
which is exactly Euler's explicit method defined in \eqref{eq:negf}.

\section{Universal gradient descent and generalized gradient flow}
\label{sec:generalized}
Consider the Generalized Gradient Flow, which combines the ideas of \eqref{eq:mixed_ode} and \eqref{eq:negf},
\begin{align}
\begin{split}
\dot y(t) &= -\nabla f(x(t)) - \nabla \Omega (x(t))  \\
x(t) &= \nabla d^*(y(t))\\
x(0) &= x_0,\; y(0) = \nabla d(x_0).
\end{split}
\tag{GGF} \label{eq:ggf}
\end{align}
We can write its ODE counterpart, called the "Generalized ODE",
\begin{align}
\begin{split}
\dot y(t) &= g(x(t)) + \omega(x(t)) \\
x(t) &= \nabla d^*(y(t)), \\
x(0) &= x_0, \; y(0) = \nabla d(x_0).
\end{split}
\tag{GODE} \label{eq:gode}
\end{align}
where $g = -\nabla f$, with $f$ a smooth convex function, $d$ a strongly convex and semi-continuous distance generating function and $\omega = -\nabla \Omega$, where $\Omega$ is a simple convex function. If  $\Omega$ is not differentiable we can consider the corresponding differential inclusion as presented in Section \ref{ssec:nonsmooth_gradflow}. Here we focus on \eqref{eq:gode} and \eqref{eq:ggf} to highlight the links with integration methods. The discretization of this ODE is able to generate many algorithms in many different settings. For example, consider the IMEX explicit-implicit Euler's method,
\[
\frac{y_{k+1}-y_{k}}{h} = g(x_k) + \omega (\nabla d^*(y_{k+1})), \quad x_{k+1} = \nabla d^*(y_{k+1}),
\]
which can be decomposed into three steps,
\begin{align}
z_{k+1} & = y_k + h g \left( x_k \right) & \text{(Gradient step in  dual space)}, \nonumber\\
y_{k+1} & = \prox_{h(\Omega\circ\nabla d^*)}\left( z_{k+1} \right) & \text{(Projection step in  dual space)},\label{eq:generalized_euler_discretization}\\
x_{k+1} & = \nabla d^*(y_{k+1}) & \text{(Mapping back in primal space)}. \nonumber
\end{align}
Now consider the universal gradient method scheme presented by \citet{nesterov2015universal}:
\[
x_{k+1} = \arg\min_x \scalar{\nabla f(x_k)}{x-x_{k}} + \Omega(x) + \mathcal{B}_d(x,x_k).
\]
Again we can show that both recursions are the same: if we write the first optimality condition,
\BEAS
0 & = & \nabla_x\left(h\scalar{\nabla f(x_k)}{x-x_{k}} + h\Omega(x) + \mathcal{B}(x,x_k)\right)\big|_{x=x_{k+1}} \\
& = & hg(x_k) + h \partial \Omega(x_{k+1}) + \nabla d(x_{k+1}) - \nabla d(x_k) \\
& = & \underbrace{hg(x_k) - y_k}_{=z_{k+1}} + h \partial \Omega(\nabla d^* (y_{k+1})) - y_{k+1}.
\EEAS
We thus need to solve the non-linear system of equations
\[
y_{k+1} = z_{k+1} + h \partial \Omega(\nabla d^* (y_{k+1})),
\]
which is equivalent to the projection step \eqref{eq:generalized_euler_discretization}. Then we simply recover $x_{k+1}$ by applying $\nabla d^*$ on $y_{k+1}$.

\section{Conclusion and future works}
\label{sec:conclusion}
We connected several optimization algorithms to multi-step integration methods in numerical analysis. By using the theory of linear multi-step methods on the basic gradient flow equation, we recover Polyak's and Nesterov's method using some optimality arguments.
This provides an intuitive interpretation for the design of these methods, with optimal step sizes giving a direct explanation of the acceleration phenomenon. Our approach generalizes to more structured problems by introducing the appropriate integration method and/or looking at a generalized gradient flow equation that takes into account the geometry of the problem. 

We described a simple interpretation of the acceleration phenomenon, but our analysis is still restricted to quadratic problems. The study of $G$-stability \citep{dahlquist1978g,butcher2006thirty} may generalize our approach to smooth strongly convex functions. For the non-strongly convex case, the dependence in $k$ of Nesterov's algorithm makes its links with integration methods less clear. 

Finally the study of the links between optimization and numerical methods may provide new algorithms for both fields. Runge-Kutta methods (another way to integrate differentials equations) may lead to newer algorithms in optimization. Conversely, Nesterov's algorithm may lead to new integration methods to integrate the gradient flow equation of a convex function.

\section*{Acknowledgment}
The research leading to these results has received funding from the European Union's Seventh Framework Programme (FP7-PEOPLE-2013-ITN) under grant agreement n$^\text{o}$ 607290 SpaRTaN, as well as support from ERC SIPA and the chaire {\em \'Economie des nouvelles donn\'ees} with the {\em data science} joint research initiative with the {\em fonds AXA pour la recherche}.

{
	\small
	\bibliography{Integration_Methods_and_Accelerated_Optimization_Algorithms}
}

\end{document}